\documentclass[
10pt]{amsart}
\title[]{$\mbox{\boldmath{$C^1$}}$-umbilics with arbitrarily high indices}
\date{\today}

\usepackage[dvips]{graphicx} 
\usepackage{verbatim,enumerate}
\usepackage{amssymb}
\usepackage{times}

\usepackage[usenames]{color}
\usepackage{amsthm}
\theoremstyle{plain}
 \newtheorem{theorem}{Theorem}[section]
 \newtheorem*{theorem*}{Theorem}

 \newtheorem*{lemma*}{Lemma}
 \newtheorem{proposition}[theorem]{Proposition}
 \newtheorem{fact}[theorem]{Fact}
 \newtheorem*{fact*}{Fact}

 \newtheorem{lemma}[theorem]{Lemma}
 \newtheorem{corollary}[theorem]{Corollary}
\theoremstyle{remark}

 \newtheorem*{remark*}{Remark}
 \newtheorem*{ack}{Acknowledgements}
 \newtheorem{example}[theorem]{Example}
\numberwithin{equation}{section}

\newcommand{\Z}{\boldsymbol{Z}}
\newcommand{\R}{\boldsymbol{R}}
\newcommand{\C}{\boldsymbol{C}}

\renewcommand{\phi}{\varphi}
\renewcommand{\epsilon}{\varepsilon}
\newcommand{\op}{\operatorname}

\newcommand{\mb}[1]{{\mathbf #1}}

\newcommand{\pmt}[1]{{\begin{pmatrix} #1  \end{pmatrix}}}
\newcommand{\vmt}[1]{{\begin{vmatrix} #1  \end{vmatrix}}}
\newcommand{\dy}{\displaystyle}

\author{Naoya Ando}
\address[Ando]{%
Graduate School of Science and Technology, Kumamoto University
2-39-1 Kurokami, Kumamoto 860-8555 Japan
}
\email{ando@sci.kumamoto-u.ac.jp}

\author{Toshifumi Fujiyama}
\address[Fujiyama]{%
Department of Mathematical Informatics, 
Graduate School of Information Science and Technology, University
of Tokyo}
\email{toshifumi\_fujiyama@mist.i.u-tokyo.ac.jp
}

\author{Masaaki Umehara}
\address[Umehara]{%
Department of Mathematical and Computing Sciences
Tokyo Institute of Technology
2-12-1-W8-34, O-okayama Meguro-ku
Tokyo 152-8552 Japan
}
\email{umehara@is.titech.ac.jp}


\thanks{
The first author was partly supported by the Grant-in-Aid for Young
Scientists (B) 24740048, Japan Society for the Promotion of Science.
The third author was partly supported by the Grant-in-Aid for
Scientific Research (A) 262457005, Japan Society for the Promotion of
Science.
}

\begin{document}
\begin{abstract}
In this paper, the existence of $C^1$-umbilics 
with arbitrarily high indices is shown.
This implies that 
more than $C^1$-regularity
is required to prove Loewner's conjecture. 
\end{abstract}
\maketitle

\section{Introduction}

The {\it 
index} of an isolated umbilic  on a given regular 
surface is the index of the curvature line flow
of the surface at that point, which takes values
in the set of half-integers.
{\it Loewner's conjecture} asserts that any isolated 
umbilic on an immersed surface must have index at most $1$.
{\it Carath\'eodory's conjecture} asserts 
the existence of at least two umbilics on an  
immersed sphere in $\R^3$, which 
follows immediately from Loewner's conjecture.
Although this problem was investigated mainly on 
real-analytic surfaces, 
several geometers recently became interested in
non-analytic cases (cf. \cite{A3,B,GH,GMS,SX1}).
In particular, 
Smyth-Xavier \cite{SX1} observed that
Enneper's minimal surface is inverted 
to a branched sphere 
such that the index of
the curvature line flow
at the branch point 
is equal to two.
Bates \cite{B}
found that the graph of
the function
\begin{equation}\label{eq:B}
B(x,y):=2+\frac{x y}{\sqrt{1+x^2}\sqrt{1+y^2}}
\end{equation}
has no umbilics on $\R^2$
and inversion of it
gives a genus zero surface
without self-intersections, which
is differentiable at the image of
infinity under that inversion.
Ghomi-Howard  \cite{GH} gave 
similar examples of genus zero surfaces 
using inversion.
Moreover, they showed that Carath\'eodory's 
conjecture for closed convex surfaces 
can be reduced to the problem of existence of 
umbilics of certain entire graphs over $\R^2$.
A brief history of 
Carath\'eodory's conjecture
and recent developments are written also
in \cite{GH}.
Recently, Guilfoyle-Klingenberg
\cite{GK1} and \cite{GK2}
gave an approach to proving the 
Caratheodory and Loewner
conjecture in the smooth case.

Let $P:U\to \R^3$ be a $C^1$-immersion 
defined on an open subset $U$
of $\R^2$ such that $P$ is $C^\infty$-differentiable
on $U\setminus \{q\}$ and not $C^2$-differentiable
at $q$.
Then the point $q\in U$ is called a
{\it $C^1$-umbilic} if
the umbilics of $P$ on $U\setminus \{q\}$
do not accumulate to $q$.
At that point $q$, we can compute
the index of the curvature line flow of $P$.
In this paper, we prove
the following assertion.

\begin{theorem}\label{thm:C1}
Let $U_1(\subset \R^2)$ be the unit disk
centered at the origin. 
For each positive integer $m$,
there exists a $C^1$-function
$
f:U_1\to \R
$
satisfying the following properties$:$
\begin{enumerate}
\item[{\rm (1)}] $f$ is real-analytic on 
$U_1^*:=U_1\setminus\{(0,0)\}$,
\item[{\rm (2)}] $(0,0,f(0,0))$
is a $C^1$-umbilic of the graph
of $f$ with index $1+(m/2)$.
\end{enumerate}
\end{theorem}

\medskip 
It should be remarked that
the inversion of the graph of 
Bates' function $B(x,y)$
has a differentiable umbilic of index $2$
although not of class $C^1$
(see Example \ref{ex:B0}). 
It was classically known that
curvature line flows
are closely related to
the eigen-flows of the Hessian matrices
of functions (see Appendix A).
As an application of the above result, 
we can show the following: 

\begin{corollary}\label{cor:C1}
For each $m(\ge 1)$,
there exists a $C^1$-function
$
\lambda:U_1\to \R
$
satisfying
\begin{enumerate}
\item $\lambda$ is real-analytic on $U_1^*$, and 
\item the eigen-flow of the
Hessian matrix of $\lambda$ has an
isolated singular point 
$(0,0)$ with index $1+(m/2)$.
\end{enumerate}
\end{corollary}

When we 
consider the eigen-flow of the 
Hessian matrix of $f$,
it is well-known that 
the index of the flow at an
isolated singular point
is equal to half of the index of the vector field 
\begin{equation}\label{eq:id0}
d_f:=2f_{xy}\frac{\partial}{\partial x}
+(f_{yy}-f_{xx})\frac{\partial}{\partial y}.
\end{equation}
In addition, if $o:=(0, 0)$ is an isolated singular point of the
eigen-flow of the Hessian matrix of $f$, then its index is equal to
$1+{\rm ind}_o (\delta_f )/2$ (see Appendix B), 
where ${\rm ind}_o (\delta_f
)$ is the index of the vector field
\begin{equation}\label{eq:id1}
\delta_f:=2(rf_{r\theta}-f_\theta)
\frac{\partial}{\partial x}+
(-r^2f_{rr}+rf_r+f_{\theta\theta})
\frac{\partial}{\partial y},
\end{equation}
at $o$, and $x=r\cos \theta$, $y=r\sin \theta$.
In order to prove the above theorem,
we introduce vector fields 
$D_f$ and $\Delta_f$
analogous to $d_f$ and
$\delta_f$, respectively
(cf. Propositions \ref{thm:xy} and \ref{thm:rt}),
and prove the theorem by
computing the index of $\Delta_f$ 
at infinity for each of 
the functions (cf. Section 5)
\begin{equation}\label{eq:tanh}
(f=)f_m(r,\theta):=
1+\tanh \left(r^a \cos m \theta\right)
\quad (0<a<1/4,\,\,m=1,2,\dots).
\end{equation}
In addition, we give an alternative proof of Theorem 1.1
without use of inversion, by an explicit example of
$\lambda$ (cf. (6.1)) satisfying (1) and (2) of Corollary 1.2
(see Section 6).

\section{The regularity of the inversion}
Let $R$ be a positive number.
Consider a function
$
f:\R^2\setminus \Omega_R\to \R,
$
where
\begin{equation}\label{eq:Omega}
\Omega_R:=\{(x,y)\in \R^2\,;\, \sqrt{x^2+y^2}\le R\}.
\end{equation}
Then
$
F=\left(x,
y,
f(x,y)\right)
$
gives a parametrization of the graph of $f$.
The inversion of $F$
is given by
$
F/(F\cdot F),
$
where the dot denotes the inner product on $\R^3$.
We consider the following coordinate change
\begin{equation}\label{eq:c-xy}
x=\frac{u}{u^2+v^2},\qquad y=\frac{v}{u^2+v^2}.
\end{equation}
Then
\begin{equation}\label{eq:A}
\Psi_f:=
\frac{1}{\rho^2\hat f^2+1}
(u,v,\rho^2\hat f),\qquad
\hat f(u,v):=f\left(\frac{u}{\rho^2},\frac{v}{\rho^2}\right)
\end{equation}
gives a parametrization of the inversion,
where 
$
\rho:=\sqrt{u^2+v^2}.
$
The map 
$\Psi_f$ 
is defined on the domain
\begin{equation}\label{eq:D}
U_{1/R}^*:=U_{1/R}\setminus\{o\}\quad
\left(
U_{1/R}:=\left\{(u,v)\in \R^2\,;\, \sqrt{u^2+v^2}< \frac1R
\right\}
\right),
\end{equation}
where $o:=(0,0)$.
If we set
\begin{equation}\label{eq:xy}
x=r \cos \theta,\qquad y= r\sin \theta,
\end{equation}
where $r>0$, then
\eqref{eq:c-xy} yields 
\begin{equation}\label{eq:uvxy}
\rho=\frac1{r},\qquad
u=\rho \cos \theta,\qquad v=\rho \sin \theta.
\end{equation}
In particular, the angular parameter 
is common in the $xy$-plane and the $uv$-plane.

\begin{proposition}\label{prop:diff}
Let $f:\R^2\setminus \Omega_R\to \R$
be a $C^\infty$-function
such that $f/r$ is bounded.
Then the inversion
$\Psi_f:U^*_{1/R}\to \R^3$ 
can be continuously extended to $(0,0)$.
Moreover, if 
\begin{equation}\label{eq:fr0}
\left|\frac{f^2 -2rf f_r}{r^2}\right|<1  \qquad (r>R),
\end{equation}
then the image 
of $\Psi_f=(X,Y,Z)$ 
can be locally expressed as 
the graph of a function $Z=Z_f(X,Y)$
on a neighborhood of $(0,0)$
in the $XY$-plane.
Under the assumption \eqref{eq:fr0}, 
the function $Z_f(X,Y)$ is differentiable
if and only if
$$ 
\lim_{r\to \infty}\frac{f}r=0.
$$ 
\end{proposition}

\begin{proof}
We can write
\begin{equation}\label{eq:wr}
\Psi_f(u,v)
=
\frac1{1+\phi(u,v)^2}
\left(
u,v,\phi(u,v)\sqrt{u^2+v^2}
\right),
\end{equation}
where
\begin{equation}\label{eq:phi}
\phi(u,v)=\sqrt{u^2+v^2}\hat f(u,v)
=\frac{f(x,y)}{r}.
\end{equation}
Since $f/r$ is bounded, the function
$\phi$ is bounded on $U^*_{1/R}$.
Thus, one can prove
$\lim_{\rho\to 0}\Psi_f=(0,0,0)$
using \eqref{eq:wr},
that is, $\Psi_f(u,v)$ can be 
continuously extended to $(0,0)$.
We denote by 
$
\Pi:\R^3 \ni (x,y,z) \mapsto (x,y)\in \R^2
$
the orthogonal projection.
By setting
$$
\psi(\rho,\theta):=\frac{\rho}
{1+\phi(\rho \cos \theta,\rho \sin \theta)^2},
$$
it holds that 
\begin{equation}\label{eq:Pi}
\Pi\circ \Psi_f(u,v)
=\biggl(\psi(\rho,\theta)\cos \theta,
\psi(\rho,\theta)\sin \theta\biggr).
\end{equation}
Since
$
\hat f(\rho \cos \theta,\rho \sin \theta)
=f(\cos \theta/\rho,\sin \theta/\rho),
$
we have
$$
\phi_\rho
= f-r f_r.
$$
In particular, it holds that
$$
\psi_\rho
=\frac{1-(f^2-2r ff_r)/r^2}{(1+f^2/r^2)^2}.
$$
By \eqref{eq:fr0},
there exists $\epsilon>0$ 
such that
$\rho\mapsto \psi(\rho,\theta)$ 
$(|\rho|\le \epsilon)$
is a monotone increasing function
for each $\theta$.
Thus, by \eqref{eq:Pi},
we can conclude that
$\Pi\circ \Psi_f:\overline{U_{\epsilon}}\to \R^2$
is an injection.
Since a continuous bijection 
from a compact space to a Hausdorff space
is a homeomorphism,
the inverse map
$
G:\Omega\to U_{\epsilon}
$
of $\Pi\circ \Psi_f|_{U_\epsilon}$
is continuous,
where
$\Omega$ is a neighborhood of
the origin of the $XY$-plane in $\R^3$.
Then the graph of
\begin{equation}\label{eq:zz}
Z_f\,\,\left(=\frac{\rho \phi}{1+\phi^2}\right)=
\frac{\phi(G(X,Y))\rho(G(X,Y))}{1+\phi(G(X,Y))^2}
\end{equation}
coincides with the image of $\Psi_f=(X,Y,Z)$ 
around $(0,0,0)$.
Then
$$
X=\frac{u}{1+\phi^2},\quad Y=\frac{v}{1+\phi^2},
\quad
Z=\frac{\rho\phi}{1+\phi^2}.
$$
Since $\rho\to 0$ as
$(X,Y)\to (0,0)$, 
we obtain
\begin{equation}\label{eq:limZ2}
\lim_{(X,Y)\to (0,0)}\frac{Z_f(X,Y)}{\sqrt{X^2+Y^2}}
=\lim_{(X,Y)\to (0,0)}
\frac{\phi\rho}{\sqrt{u^2+v^2}}=
\lim_{\rho\to 0}\phi=\lim_{r\to \infty}\frac{f}{r},
\end{equation}
proving the last assertion.
\end{proof}

\begin{corollary}\label{cor:diff}
Suppose that $f:\R^2\setminus \Omega_R\to \R$
is a bounded $C^\infty$-function
satisfying 
\begin{equation}\label{eq:fr}
\lim_{r\to \infty}\frac{f_r}r=0.
\end{equation}
Then the inversion
$\Psi_f:U^*_{1/R}\to \R^3$ 
can be continuously extended  to $(0,0)$.
Moreover, the
image of $\Psi_f$ is locally a
graph which 
is differentiable at $(0,0)$.
\end{corollary}

\begin{example}\label{ex:B0}
Bates' example (cf. \eqref{eq:B})
mentioned in the introduction
is differentiable. In fact, $B(x,y)$
is bounded and
$B_r/r$ converges to zero as $r\to \infty$.
However, the inversion of $(x, y, B(x, y))$ is 
not $C^1$.
In fact, 
the unit normal vector
field of the graph of $B$ is 
not continuously extended to 
the point at infinity.
Since the inversion preserves the angle,
the unit normal vector field of 
its inversion cannot be
continuously extended to $(0, 0, 0)$. 
\end{example}

\begin{example}\label{ex:GH}
Ghomi-Howard \cite{GH} gave an example
\begin{equation}\label{eq:GH}
f_{GH}=1+\lambda \frac{1+x+y^2}{\sqrt{1+(x+y^2)^2}}
\qquad (\lambda>0).
\end{equation}
The graph of $f_{GH}$ is an umbilic free
(see Example \ref{ex:B2} in Section 3).
The function $f_{GH}$ is bounded.
In addition, since
$(f_{GH})_r$ is bounded, \eqref{eq:fr} is
obvious. Therefore,
as pointed out in \cite{GH},
the inversion of $(x, y,f_{GH}(x,y))$ 
is differentiable. 
However, it is not a $C^1$-map.
In fact, the limit of
the unit normal
 vector field along $y=0$
of the graph of $f_{GH}$ 
is not equal to that along $x+y^2 =0$
at the point at infinity.
\end{example}

Next, we give a condition
for $\Psi_f$ to be 
extendable as a $C^1$-map to $(0,0)$.

\begin{proposition}\label{prop:a}
Suppose that $f:\R^2\setminus \Omega_R\to \R$
is a bounded $C^\infty$-function satisfying
$$ 
\mbox{\rm (a)}\,\,\, 
\lim_{r\to \infty}f_r=0,\qquad 
\mbox{\rm (b)}\,\,\, 
\lim_{r\to \infty}f_\theta/r=0. 
$$ 
Then $\Psi_f=(X,Y,Z)$ can be
extended to $(0,0)$ as a $C^{1}$-map.
Moreover, 
the map
$(u,v)\mapsto (X(u,v),Y(u,v))$ 
is a $C^1$-diffeomorphism
from a neighborhood
of the origin in the $uv$-plane
onto a neighborhood of the origin
in the $XY$-plane.
\end{proposition}

To prove this, we prepare the 
following lemma.

\begin{lemma}\label{lem:A}
The conditions $($a$)$ and $($b$)$
in Proposition \ref{prop:a}
are equivalent to the following two
conditions, respectively$:$
$$
\mbox{\rm (1)}\,\,\, 
\lim_{\rho\to 0}\rho^2\hat f_\rho=0,
\qquad
\mbox{\rm (2)}\,\,\, 
\dy\lim_{\rho\to 0}\rho \hat f_\theta=0.
$$ 
\end{lemma}

\begin{proof}
The equivalency of (2) and (b) is
obvious.
The equivalency of (1) and (a)
follows from the identity
$
\hat f_\rho
=-{f_r}/{\rho^2}. 
$
\end{proof}

\begin{proof}[Proof of 
Proposition \ref{prop:a}]
We see by Corollary 2.2 that $\Psi_f$ can be extended
to $(0, 0)$ as a differentiable map and 
the map $(u, v)\mapsto (X(u, v), Y(u, v))$ 
is a homeomorphism from a neighborhood of 
$(0, 0)$ onto a
neighborhood of $(0, 0)$.
We set
\begin{equation}\label{eq:hk}
h:=\rho^2 \hat f(=\rho \phi),\quad k:=(\rho \hat f)^2(=\phi^2).
\end{equation}
By \eqref{eq:A}, we can write
\begin{equation}\label{eq:Psi}
\Psi_f=(X,Y,Z)=
\frac{1}{k+1}(u,v,h).
\end{equation}
To show that $\Psi_f$ is a $C^1$-map
at $(0,0)$,
it is sufficient to show that
$h,k$ are $C^1$-functions.
Since $h$ and $k$ are $C^\infty$-functions
on $U_{1/R}^*$,
they satisfy
\begin{align}\label{eq:h0}
h_u&=
\rho  \biggl((2 {\hat f} + \rho {\hat f}_\rho) \cos \theta
-{\hat f}_\theta \sin \theta\biggr), \quad
h_v=
\rho  
\biggl((2 {\hat f} + \rho {\hat f}_\rho) \sin \theta 
+{\hat f}_\theta \cos \theta \biggr),\\
\label{eq:k0}
k_u&=
2 {\hat f} \rho  \biggl(\cos \theta ({\hat f}+\rho {\hat f}_\rho)
-{\hat f}_\theta \sin \theta\biggr),\quad
k_v=
2 {\hat f} \rho  \biggl(\sin \theta ({\hat f}+\rho {\hat f}_\rho)
+{\hat f}_\theta \cos \theta\biggr)
\end{align}
on $U_{1/R}^*$.
Using (1), (2) in Lemma \ref{lem:A},
\eqref{eq:h0} and \eqref{eq:k0},
one can easily see that
\begin{equation}\label{eq:00}
\lim_{\rho\to 0}h_u=\lim_{\rho\to 0}h_v
=\lim_{\rho\to 0}k_u=\lim_{\rho\to 0}k_v=0,
\end{equation}
which shows that
$\Psi_f$ extends to $(0,0)$
as a $C^1$-map.
By  \eqref{eq:Psi} and
\eqref{eq:00}, we have
$$
X_u(0,0)=1,\quad X_v(0,0)=0,\quad
Y_u(0,0)=0,\quad Y_v(0,0)=1.
$$
Thus the
second assertion follows from
the inverse mapping theorem,
because the Jacobi matrix of
the map
$(u,v)\mapsto (X(u,v),Y(u,v))$ 
is regular at $(0,0)$.
\end{proof}

In Section 5, we need the following:

\begin{proposition}\label{thm:a}
Let $f:\R^2\setminus \Omega_R\to \R$ be a 
bounded $C^\infty$ function
satisfying the conditions (a)
and (b) of
Proposition \ref{prop:a}.
If there exists a constant
$
0\le c<1/2
$
such that
$$
r^{1-c/2}f_r,\quad 
r^{-c/2} f_\theta,\quad 
r^{2-c}f_{rr},\quad 
r^{1-c}f_{r\theta},\quad 
r^{-c}f_{\theta\theta}
$$
are bounded on $\R^2\setminus \Omega_R$,
then the map
$
(u,v)\mapsto (X(u,v),Y(u,v))
$
is a $C^2$-map at $(0,0)$,
where $\Psi_f=(X,Y,Z)$.
\end{proposition}

We prepare the
following lemmas:

\begin{lemma}\label{lem:28}
The boundedness of the five
functions in Proposition 
\ref{thm:a} is equivalent
to the boundedness of 
the functions
\begin{equation}\label{eq:5ft}
\rho^{1+c/2}\hat f_\rho,\quad 
\rho^{c/2} \hat f_\theta,\quad 
\rho^{2+c} \hat f_{\rho\rho},\quad 
\rho^{1+c} \hat f_{\rho\theta},\quad 
\rho^c \hat f_{\theta\theta} 
\end{equation}
on $U\setminus \{(0,0)\}$,
where $U$ is 
a sufficiently small neighborhood
of $(0,0)$.
\end{lemma}

\begin{proof}
Differentiating 
$\hat f=\hat f(\rho \cos \theta,
\rho \sin \theta)$ 
by $\rho$, we have
$\rho \hat f_{\rho}=-r f_r$ and
$\rho^2 \hat f_{\rho\rho}=2r f_r+r^2 f_{rr}$.
Using these relations,
the assertion can be easily checked.
\end{proof}

\begin{lemma}\label{lemma:eq3}
Suppose that the five functions in \eqref{eq:5ft} 
are
bounded on $U\setminus \{ (0, 0)\}$.
Then 
$\rho^{2c}k_{uu},\,\,\rho^{2c} k_{uv}$ and
$\rho^{2c}k_{vv}$
are also bounded on $U\setminus \{ (0, 0)\}$,
where $k$
is the function given in \eqref{eq:hk}.
\end{lemma}

\begin{proof}
In fact, each of
$
k_{uu},k_{uv},k_{vv}
$
is written 
as  a linear combination
of
$$
1, \quad\rho \hat f_\rho, \quad\hat f_\theta,
\quad
(\rho\hat f_\rho)^2,\quad 
\rho\hat f_\rho \hat f_\theta,\quad 
\hat f_\theta^2,
\quad
\rho^{2} \hat f_{\rho\rho},\quad 
\rho \hat f_{\rho\theta},\quad 
\hat f_{\theta\theta} 
$$
with coefficients that are 
bounded functions. 
For example,
\begin{align*}
k_{uv}&=
\sin 2 \theta  \left(\rho ^2 \hat f_\rho^2+\hat f
 \bigl(\rho ^2 \hat f_{\rho\rho}+3 \rho  
\hat f_\rho-\hat f
_{\theta\theta}\bigr)
-\hat f_\theta^2\right)\\
& \phantom{aaaaa}+2 \cos 2 \theta  \left(\hat f_\theta 
\bigl(\rho  \hat f_\rho+\hat f\bigr)+\rho  \hat f \hat 
f_{\rho\theta}\right).
\end{align*}
Thus, we get the assertion.
\end{proof}

\begin{proof}[Proof of Proposition \ref{thm:a}]
By Lemmas \ref{lem:28} and
\ref{lemma:eq3}, the fact that
$2c<1$ yields that
\begin{equation}\label{eq:lim-k}
\lim_{\rho\to 0}\rho k_{uu}=\lim_{\rho\to 0}\rho k_{uv}
=\lim_{\rho\to 0}\rho k_{vv}=0.
\end{equation}
Since
\begin{align*}
X_{uu}&=\frac{2 u k_u^2-2 (k+1) k_u-u (k+1) k_{uu}}{(k+1)^3},\\
X_{uv}&=
-\frac{k_v \left(-2 u k_u+k+1\right)+u (k+1) k_{uv}}{(k+1)^3},\quad
X_{vv}=
-\frac{u \left((k+1) k_{vv}-2 k_v^2\right)}{(k+1)^3},
\end{align*}
we have that
$X_{uu},X_{uv},X_{vv}$
tend to $0$ as $\rho\to 0$.
This implies that $X_u,X_v$ are $C^1$-functions.
Similarly, $Y_u,Y_v$ are also $C^1$-functions.
\end{proof}

\section{The pair of identifiers for umbilics}

Let $U$ be a domain on  $\R^2$.
Consider a flow (i.e. a $1$-dimensional foliation)
$\mathcal F$ defined on $U\setminus \{p_1,\dots,p_n\}$,
where $p_1,\dots,p_n$ are distinct points in $U$.
We are interested in the case that 
$\mathcal F$ is 
\begin{itemize}
\item the curvature line flow of an immersion $P:U\to \R^3$, 
\item the eigen-flow of a matrix-valued function 
on $U$, or 
\item the flow induced by a vector field on $U$.
\end{itemize}
We fix a simple closed smooth  
curve $\gamma:T^1\to U\setminus \{p_1,\dots,p_n\}$,
where $T^1:=\R/2\pi\Z$.
We set
$$
\partial_x:=\frac{\partial}{\partial x},\qquad
\partial_y:=\frac{\partial}{\partial y}.
$$
Then one can take a smooth vector field 
$$
V(t):=a(t) \partial_x+b(t)\partial_y
$$
along the curve $\gamma(t)$
such that $V(t)$ is a non-zero tangent vector of $\R^2$
at $\gamma(t)$ which points in the direction of
the flow $\mathcal F$.
Then the map
\begin{equation}\label{eq:hat}
\check V:T^1 \ni t \mapsto 
\frac{(a(t),b(t))}{\sqrt{a(t)^2+b(t)^2}}
\in S^1:=\{\mb x\in \R^2\,;\, |\mb x|=1\}
\end{equation}
is called the {\it Gauss map} of $\mathcal F$
with respect to the curve $\gamma$.
The mapping degree of the map
$\check V$ is called 
the {\it rotation index} of $\mathcal F$
with respect to $\gamma$ and
denoted by $\op{ind}(\mathcal F,\gamma)$,
which is a half-integer, in general.
If $\gamma$ surrounds only $p_j$,
then $\op{ind}(\mathcal F,\gamma)$ is independent
of the choice of such a curve $\gamma$.
So we call it 
the {\it $($rotation$)$ index} of the flow $\mathcal F$
at $p_j$, and it is denoted by $\op{ind}_{p_j}(\mathcal F)$.
If the flow $\mathcal F$ is generated by a
vector field $V$ defined on $U\setminus \{p_1,\dots,p_n\}$,
then $\op{ind}_{p_j}(\mathcal F)$
is an integer, and we denote it by
$\op{ind}_{p_j}(V)$.

We denote by 
$S_2(\R)$ the set of 
real symmetric $2$-matrices.
Let $U$ be a domain
in $\R^2$,
and
$$
A=\pmt{a_{11}(x,y)& a_{12}(x,y)\\
       a_{12}(x,y) & a_{22}(x,y)
}:U\to S_2(\R)
$$
a $C^\infty$-map.
A point $p\in U$ is called an 
{\it equi-diagonal point}
of $A$
if $a_{11}=a_{22}$ and
$a_{12}=0$ at $p$.
We now suppose that $p$ is
an isolated equi-diagonal point. 
Without loss of generality,
we may assume that $A$ has
no equi-diagonal points on $U\setminus\{p\}$.
Since two eigen-flows of $A$ are 
mutually orthogonal, the indices of 
the two eigen-flows of 
the $S_2(\R)$-valued function
$A$
are the same half-integer at $p$.
We denote it by $\op{ind}_p(A)$. 

It is well-known that 
for an $S_2(\R)$-valued function $A$,
the formula
\begin{equation}\label{eq:pa}
\op{ind}_{p}(A)=
\frac12\op{ind}_{p}(\mb v_A)
\end{equation}
holds, where $\mb v_A$ is the vector field
on $U$
given by
\begin{equation}\label{eq:va}
\mb v_A:=(a_{11}-a_{22})\partial_x+a_{12}\partial_y.
\end{equation}

We shall apply these facts
to the computation of the 
indices of isolated umbilics
on regular surfaces 
in $\R^3$ as follows.
Let 
$
f:U\to \R
$
be a $C^\infty$-function.
The symmetric matrices associated with
the first and the second fundamental forms of 
the graph of $f$ are given by
\begin{equation}\label{eq:I}
I:=\pmt{1+f_x^2 &  f_x f_y \\ f_xf_y & 1+f_y^2},\qquad
I\!I:=\pmt{f_{xx} & f_{xy} \\ f_{xy} & f_{yy}}.
\end{equation}
We consider 
a ${GL}(2,\R)$-valued function
\begin{equation}\label{eq:P}
P:=
\pmt{ 0 & \sqrt{1+f_x^2} \\
 -\sqrt{(1+f_x^2+f_y^2)/(1+f_x^2)} 
& {f_x  f_y}/{\sqrt{1+f_x^2}}},
\end{equation}
which satisfies the identity
$
PP^T=I
$, where $P^T$ is the transpose of $P$.
Then
$$
A_f:=P^{-1}I\!I(P^T)^{-1}=P^T(I^{-1}I\!I)(P^T)^{-1}
$$
is an $S_2(\R)$-valued function.
The umbilics of the graph of $f$
correspond to
the equi-diagonal points of $A_f$.
We show the following:

\begin{proposition}\label{prop:AB}
The symmetric matrix
$A_f(p)$ 
is proportional to the identity matrix
at $p\in U$
if and only if $p$ gives an umbilic of 
the graph of $f$. 
Moreover, if $p$ is an isolated umbilic,
then $\op{ind}_p(A_f)$ coincides with
the index of the umbilic $p$.
\end{proposition}

\begin{proof}
The first assertion follows from the definition of $A_f$.
Without loss of generality, we may assume
that $p$ coincides with  the origin $o:=(0,0)$, 
and the graph of $f$ has no
umbilics other than $o$ on $U$.
Take a sufficiently small positive
number $\epsilon>0$ so that
the circle
$
\gamma(t)=\epsilon\pmt{\cos t, \sin t}
$
($0\le t \le 2\pi$)
is null-homotopic in $U$. 

We denote by $(a_1 (t), b_1 (t))^T$ and $(a_2 (t), b_2 (t))^T$
eigen-vectors of
$I^{-1} I\!I$ 
and $A_f$ at $\gamma (t)$, respectively. We may suppose
$$
(a_1 (t), b_1 (t))P(\gamma (t))= (a_2 (t), b_2 (t))
 \qquad (0\leq t\leq
2\pi ).
$$
We set
$$
{\bf w}_i (t):=a_i (t)\partial_x + b_i (t)\partial_y 
\qquad (i=1,2).
$$
Then ${\bf w}_1$ points in one of the 
principal directions of the
graph of
$f$.
The  matrix $P(\gamma (t))$ 
takes values in
the set
\begin{equation}\label{eq:T}
\mathcal T:=\left\{\pmt{0 & x\\ -y & z}
\,;\, x,y>0,\,\, z\in \R\right\}.
\end{equation}
Since the set $\mathcal T$
is null-homotopic,
the mapping degree of 
$\check{\mb w}_1(t)$ with 
respect to the origin
is equal to that of
$\check{\mb w}_2(t)$.
Since 
the degree of 
$\check{\mb w}_2(t)$ with 
respect to $o$
coincides with 
$\op{ind}_o(A_f)$, we get the second 
assertion.
\end{proof}
\medskip

By a straightforward calculation,
one can get
the following identity:
$$
\tilde A_f:=hk^{3}A_f
=
\pmt{
f_x f_y (f_x f_yf_{xx}-2h f_{xy})+h^2 f_{yy}
& lk 
\\
lk &  k^2f_{xx}},
$$
where
$$
h:=1 + f_x^2,\quad k:=\sqrt{1 + f_x^2 + f_y^2},
\quad
l:=-h f_{xy} + f_xf_y f_{xx}.
$$
Then the coefficients of the vector field
$$
\mb v_{\tilde A_f}=v_1\partial_x+v_2\partial_y
$$ 
defined as in 
\eqref{eq:va} for $A=\tilde A_f$
are given by
\begin{align*}
&v_1=\tilde a_{11}-\tilde a_{22}=
(-1 + f_x^2) f_y^2f_{xx}  
-h f_{xx} - 
 2 hf_x f_{xy} f_y  
+ h^2 f_{yy}, \\
&v_2=\tilde a_{12}
=-k(h f_{xy} - f_xf_y f_{xx}),
\end{align*}
where $\tilde A_f=(\tilde a_{ij})_{i,j=1,2}$.
Hence, we get the following identity
$$
v_1=\frac{2f_xf_y}k v_2+h \biggl(-f_{xx} (1 + f_y^2) + (1 + f_x^2) f_{yy}\biggr).
$$
Consequently, we get the following fact
(cf. Ghomi-Howard \cite[(10)]{GH}):
 
\begin{fact}\label{prop:gh}
The graph of the function $z=f(x,y)$ 
defined on $U$ has an umbilic at 
$p\in U$ if and only if the functions
$$
d_1(x,y):=(1 + f_x^2) f_{xy} - f_xf_y f_{xx},
\quad 
d_2(x,y):= (1 + f_x^2) f_{yy}-f_{xx} (1 + f_y^2)
$$
both vanish at $p$.
\end{fact}

We consider the vector field
$$
D_f:=d_1 \partial_x
+d_2 \partial_y
$$
defined on the domain $U$ in the $xy$-plane.
Suppose that $p$ is a  zero of $D_f$.
The following assertion holds:

\begin{proposition}\label{thm:xy}
If $p$ gives an isolated umbilic
of the graph of $f$, 
then half of the index of the vector field
$D_f$
at $p$ coincides with the index of 
the umbilic $p$.
\end{proposition}

\begin{proof}
The half of the index of the vector field
$$
X:= -\mb v_{\hat A_f}=
({2f_xf_yd_1}-h d_2)\partial_x+
k d_1\partial_y
$$
at $p$
is equal to $\op{ind}_{p}(\tilde A_f)$.
We now set 
$$
X_s:=(\partial_x,\partial_y)
\pmt{2sf_xf_y& -1 - s f_x^2\\
\sqrt{1+s(f_x^2+f_y^2)}
 & 0}
\pmt{d_1\\ d_2}
\qquad (0\le s\le 1).
$$
Then $X=X_1$ and 
$X_0=-d_2\partial_x+d_1\partial_y$,
and the rotation index of
$X_s$ at $p$
does not depend on $s\in [0,1]$.
Since the rotation index of
$D_f=(d_1,d_2)$ at $p$
coincides with
that of $X_0$,
we can conclude that
$X$ has the same rotation 
index as $D_f$ at $p$.
\end{proof}

We call $d_1,d_2$
the {\it Cartesian umbilic identifiers} 
of the function $f$.

\begin{example}
For a function 
$
f(x,y):=\op{Re}(z^3)=x^3 - 3 x y^2
$
($z=x+iy$),
the Cartesian umbilic identifiers
are given by
$
d_1=-6y \phi_1,\,\, d_2=
-6x \phi_2,
$
where
$$
\phi_1:=-9 x^4+9 y^4+1,\qquad
\phi_2:=
9 x^4+18 x^2 y^2+9 y^4+2.
$$
Since $\phi_i$ ($i=1,2$)
are positive at the origin $(0,0)$,
the vector field $D_f$
can be continuously deformed into 
the vector field 
$-y\partial_x-x\partial_y$
preserving the property that the origin
is an isolated zero.
Thus $D_f$ is of index $-1$,
and the graph of the function $f$ has an 
isolated umbilic of
index $-1/2$ at the origin.
\end{example}

\begin{example}\label{ex:B2}
Bates' function $B(x,y)$
has no umbilics since $d_1>0$ on $\R^2$.
On the other hand, the 
identifier $d_1$ with respect to
Ghomi-Howard's function $f_{GH}(x,y)$
in \eqref{eq:GH} 
vanishes if and only if  $y=0$ or $x=-y^2$.
Since $d_2$ never vanishes on these two
sets, the graph of $f_{GH}$ also has 
no umbilics on $\R^2$.
\end{example}

\section{The pair of polar identifiers for umbilics}
Let $U$ be a domain in the $xy$-plane, and
$
f:U\to \R
$
a $C^\infty$-function.
Let $(r,\theta)$ be
the polar coordinate system
associated to $(x,y)$ as in \eqref{eq:xy}.
Then 
$$
F(r,\theta):=(r \cos \theta, r\sin \theta, 
f(r \cos \theta,r \sin \theta))
$$
gives a parametrization of the graph of $f$ 
with the unit normal vector 
$$
\nu:=\frac1{\sqrt{f_\theta^2+r^2\left(1+f_r^2\right)}}
\biggl(f_\theta \sin \theta -r f_r\cos \theta,
-r f_r\sin \theta 
-f_\theta \cos \theta,r\biggr).
$$
Then
$
\hat I:=
\pmt{
 1+f_r^2 & f_r f_\theta \\
 f_r f_\theta & r^2+f_\theta^2 
}
$
is the symmetric matrix
consisting of the coefficients of
the first fundamental form of $F$.
If we set
$$
Q=
\pmt{
 0 & \sqrt{1+f_r^2} \\
 -\sqrt{f_\theta^2+r^2\left(1+f_r^2\right)}
/\sqrt{1+f_r^2} &
f_r f_\theta/\sqrt{1+f_r^2} 
},$$
then $Q Q^{T}=\hat I$.
The symmetric matrix
consisting of the coefficients of
the second fundamental form 
is given by
$$
\widehat{I\!I}
:=
\frac1{\sqrt{f_\theta^2+r^2\left(1+f_r^2\right)}}
\pmt{
 rf_{rr}  & r f_{r\theta}-f_\theta \\
 rf_{r\theta}-f_\theta & r (f_{\theta\theta}+r f_r)}.
$$
Then the symmetric matrix
$$
B_f
=Q^{-1}\widehat{I\!I}(Q^{-1})^T
=Q^{T}(\hat I^{-1}\widehat{I\!I})(Q^{T})^{-1}
$$
satisfies
$$
\tilde B_f=\hat h \hat k^3 B_f
=\pmt{
rf_r^2 f_\theta^2 f_{rr} +\hat hf_r  
\left(-2 rf_\theta f_{r\theta} +2 
f_\theta^2+r^2\hat h \right)
+r\hat h^2f_{\theta\theta} 
 & 
\hat l   \hat k\\
\hat l   \hat k
& r \hat k^2f_{rr} 
},
$$
where
$$
\hat h:=1+f_{r}^2,\quad
\hat k:=\sqrt{f_\theta^2+r^2\left(1+f_r^2\right)}, \quad
\hat l
:=
f_{\theta} \left(\hat h+r f_rf_{rr}\right)
-r \hat h f_{r\theta}. 
$$
The following holds. 

\begin{proposition}\label{prop:hatAB}
The symmetric matrix
$\tilde B_f(p)$ is proportional 
to the identity matrix at 
$p\in U\setminus\{o\}$
if and only if $p$ gives 
an umbilic of the graph of $f$. 
Moreover, if $o$ is an isolated umbilic
of the graph of $f$, 
then the index of the umbilic
at $o$ is equal to
$1+\op{ind}_o(\tilde B_f)$.
\end{proposition}

\begin{proof}
The first assertion follows from the above discussions.
So we now prove the second assertion.
Suppose $o$ is an isolated umbilic.
We take a simple closed smooth curve 
$\gamma(t)$ ($0\le t\le 2\pi$)
in the $xy$-plane
which
surrounds the origin $o$ anti-clockwisely,
and does not surround any other umbilics.
Let 
$\mb w_1:[0,2\pi]\to \R^2$
be a vector field 
along $\gamma$
such that
$\mb w_1(t)$
is an eigen-vector
of the matrix $I^{-1}I\!I$ 
at $\gamma(t)$ for each $t\in [0,2\pi]$.
Since
$$
\partial_r=\cos \theta \partial_x+
\sin \theta \partial_y,\qquad
\partial_\theta=-r \sin \theta \partial_x+
r\cos \theta \partial_y,
$$
we have that
$$
(\partial_r,\partial_\theta)=(\partial_x,\partial_y)
T_0,
\qquad
T_0:=\pmt{\cos \theta& -r \sin \theta 
\\ \sin \theta & r\cos \theta}.
$$
Then, it holds that
$$
\hat I^{-1}\widehat{I\!I}=
(T_0)^{-1}(I^{-1}I\!I) T_0.
$$
In particular,  
$$
\mb w_2(t):=T_0(\gamma(t))^{-1} \mb w_1(t)
\qquad (0\le t \le 2\pi)
$$
gives an eigen-vector of
the matrix $\hat I^{-1}\widehat{I\!I}$ at $\gamma(t)$.
Let $T_s:U\to \op{\it GL}(2,\R)$ ($0\le s \le 1$)
be a map defined by
$$
T_s:=\pmt{\cos \theta& -(r(1-s)+s) \sin \theta \\ 
\sin \theta & (r(1-s)+s)\cos \theta}
\qquad (0\le s \le 1).
$$
Then it gives a continuous deformation of
$T_0$ to the rotation matrix $T_1$.
Since the winding number
of the curve $\gamma(t)$ with respect to
the origin $o$ is equal to $1$,
the difference between the rotation indices of
$\mb w_1$ and $\mb w_2$ is equal to $1$.
Since the eigen-flow of the symmetric matrix 
$\tilde{B}_f$ is associated with that of the 
matrix $\hat I^{-1}\widehat{I\!I}$
by $Q$, the fact that $Q$ 
takes values in the set $\mathcal T$
in Section 3 yields that 
the index of the umbilic $o$
is equal to $1+\op{ind}_{o}(\tilde B_f)$.
\end{proof}

We now set
$$
\delta_1:=-\frac{\tilde b_{12}}{\hat k}=
-f_\theta \left(1+f_r^2+rf_r f_{rr}\right)+
r\left(1+f_r^2\right) f_{r\theta},
$$
where $\tilde B_f=(\tilde b_{ij})_{i,j=1,2}$.
Then we have
$$
\tilde b_{11}-\tilde b_{22}=-2f_rf_\theta \delta_1+
r\left(1+f_r^2\right) \delta_2,
$$
where
$$
\delta_2:=
\left(1+f_r^2\right) (rf_r+f_{\theta\theta})
-f_{rr} \left(r^2+f_\theta^2\right)\!.
$$
Thus, as in the proof 
of Proposition \ref{thm:xy},
we get the following assertion.

\begin{proposition}\label{thm:rt}
Let $U$ be a neighborhood of the origin $o:=(0,0)$.
Let $f:U\to \R$ be a $C^\infty$-function.
Then the graph of $f$
has an umbilic at $p\in U\setminus\{o\}$
if and only if the two functions
$
\delta_1(r,\theta),\,\, \delta_2(r,\theta)
$
both vanish at $p$,
where $x=r\cos \theta$ and
$y=r\sin \theta$.
Moreover, if $o$ is an isolated umbilic, 
then half of the index of the vector field
$$
\Delta_f:=\delta_1\partial_x+\delta_2\partial_y
$$
at $o$
equals $-1+I_f(o)$,
where $I_f(o)$ is the index of
the umbilic $o$.
\end{proposition}

We call $\delta_1,\delta_2$
the {\it polar umbilic identifiers} 
of the function $f$.

\begin{example}
Consider the function
($z=x+iy$)
$$
f(x,y):=\op{Re}(z^2 \bar z )=
x^3 + x  y^2=r^3 \cos \theta.
$$
By straightforward calculations, we have
$$
\delta_1=-2 r^3 \sin \theta,
\qquad
\delta_2=-2 r^3  
\left(2-3 r^4-6 r^4 \cos2 \theta\right)\cos \theta.
$$
Since
$
2-3 r^4-6 r^4 \cos2 \theta
$
is positive
for sufficiently small $r>0$,
the vector field $\Delta_f$ can be
continuously deformed into 
the vector field
$
-\sin \theta\partial_r-\cos \theta\partial_\theta
$
preserving the property that the origin is
an isolated zero.
Thus the rotation index 
of $\Delta_f$ at $o$ is equal to $-1$,
and 
$
I_f(o)=1-1/2=1/2.
$
\end{example}

We give a modification of
Proposition \ref{thm:rt} for
the computation of index of
the curvature line flow of a surface
along an arbitrarily given
simple closed curve surrounding
the origin as follows.
Let $z=f(x,y)$ be a $C^\infty$-function
defined on $\R^2$ admitting only
isolated umbilics.
Suppose that $\gamma:\R\to \R^2$
be a $C^\infty$-map
satisfying $\gamma(t+2\pi)=\gamma(t)$
which gives a simple closed 
curve in the $xy$-plane such that
it surrounds a bounded domain
containing the origin $o$ anti-clockwisely.
Moreover, we assume  that
$\gamma(t)$ does not pass through 
any points corresponding 
to umbilics of the graph of $f$.
We denote by $I_f(\gamma)$ 
(resp. $\op{ind}_{\gamma}(\Delta_f)$)
the rotation index 
of the curvature line flow
(resp. of the vector field $\Delta_f$)
along the simple closed curve $\gamma$.
Then the formula
\begin{equation}\label{eq:general}
I_f(\gamma)=1+\frac{\op{ind}_{\gamma}(\Delta_f)}2
\end{equation}
can be proved by modifying the proof of
Proposition \ref{thm:rt}.
Suppose that there exist at most
finitely many
points
$
t=t_1,\dots,t_k \in [0,2\pi]
$
such that $\delta_1(\gamma(t))$ vanishes
at $t=t_j$.
We now assume that 
$\delta'_1(\gamma(t)):=d\delta_1(\gamma(t))/dt$ 
does not vanish at $t=t_j$ ($j=1,\dots,k$).
We set
$$
\epsilon(t_j)=
\begin{cases}
0 & (\mbox{$\delta_2(\gamma(t_j))<0$}), \\
1 & (\mbox{$\delta'_1(\gamma(t_j))>0$ and 
$\delta_2(\gamma(t_j))>0$}), \\
-1 & (\mbox{$\delta'_1(\gamma(t_j))<0$ and 
$\delta_2(\gamma(t_j))>0$}). 
\end{cases}
$$
Then, it holds that
\begin{equation}\label{eq:count}
\op{ind}_{\gamma}(\Delta_f)
=-\sum_{j=1}^k\epsilon(t_j).
\end{equation}

\begin{figure}[htb]
 \begin{center}
 \includegraphics[height=4.9cm]{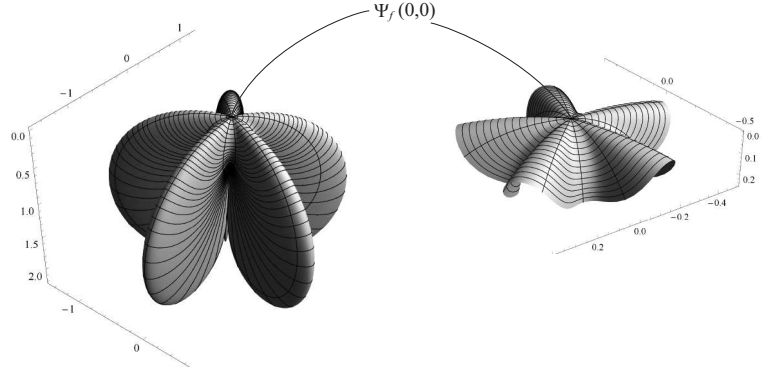}
\caption{The inversion of 
the graph $f_5$ for $a=1/5$ (left) 
and its enlarged view (right).
In these two figures,
the $z$-axis points toward the downward direction.}
\label{fig:01}
 \end{center}
\end{figure}

\section{Proof of the main theorem}
In this section, using
the function $f=f_m$
($m=1,2,3,\dots$) given
in \eqref{eq:tanh},
we prove Theorem \ref{thm:C1} and
Corollary \ref{cor:C1}
in the introduction.
More generally, we consider the
function
\begin{equation}
\label{eq:gm}
(g:=)g_m(r,\theta):=1+F(r^a \cos m \theta)
\qquad (0<a<1/4,\,\,\,m=1,2,3,\dots) ,
\end{equation}
which is defined on
$\{(r,\theta)\,;\,r>R\}$,
where $R$ is 
an arbitrarily fixed positive number,
and $F:\R\to \R$ is a bounded $C^\infty$-function
satisfying the following conditions:
\begin{enumerate}
\item[(i)] $F(x)$ is an odd function,
that is, it satisfies $F(-x)=-F(x)$,
\item[(ii)] the derivative $F'(x)$ of $F$
is a positive-valued bounded function on $\R$,
\item[(iii)]
the second derivative $F''(x)$
is a bounded function on $\R$
such that $F''(x)<0$ for $x>0$,
\item[(iv)]
there exist three constants $\alpha,\beta$ and $\gamma$
($\beta\ne 0,\,\,\gamma>0$)
such that
$$
\lim_{x\to \infty}e^{\gamma x}F'(x)=\alpha,\qquad
\lim_{x\to \infty}e^{\gamma x}F''(x)=\beta.
$$
\end{enumerate}
One can easily construct
a bounded $C^\infty$-function $F(x)$ satisfying the
properties (i-iv). For example,
one can construct
an odd $C^\infty$-function
satisfying (ii) and (iii) so that
$$
F(x)=1-e^{-x} \qquad (x\in  [M,\infty)).
$$ 
Then it satisfies also (iv).
However, to prove 
Theorem \ref{thm:C1}, we must choose
the function $F(x)$ to be
real-analytic, and
$$
F(x):=\tanh x
$$
satisfies all of the properties
required.
From now on, we shall
prove Theorem \ref{thm:C1}
and Corollary \ref{cor:C1}
using only
the above four properties of $F(x)$.

\medskip
The function $g$ can be considered 
as a $C^\infty$-function on $\R^2\setminus \Omega_R$
in the $xy$-plane for any $R>0$.
The graph of $g$ lies 
between two parallel planes
orthogonal to the $z$-axis,
and is symmetric under 
rotation by the angle $2\pi/m$
with respect to the $z$-axis
(the entire figure of the inversion of
the graph of 
$f_5$ is given in 
the left-hand side of
Figure \ref{fig:01}).
The partial derivatives of
the function $g$ are given by
\begin{align}\label{eq:diff-f}
g_r&=
a r^{a-1} c_m F'\left(r^a c_m\right),\qquad
g_\theta=
-m r^a s_m F'\left(r^a c_m\right),
\\
\nonumber
g_{rr}&=
a r^{a-2} c_m \biggl(a r^a c_m F''\left(r^a c_m\right)
+(a-1) F'\left(r^a c_m\right)\biggr), \\
\nonumber
g_{r\theta}&=
-a m r^{a-1} s_m \biggl(r^a c_m F''\left(r^a c_m\right)
+F'\left(r^a c_m\right)\biggr),
\\
\nonumber
g_{\theta\theta}&=
m^2 r^a \biggl(r^a s_m^2 F''\left(r^a c_m\right)
-c_m F'\left(r^a c_m\right)\biggr),
\end{align}
where
\begin{equation}\label{eq:cmsm}
c_m:=\cos m \theta,\qquad s_m:=\sin m\theta.
\end{equation}
Since $F(x)$ is a bounded function,
$g$ is bounded and 
satisfies \eqref{eq:fr}, since $a<2$.
Therefore,
the inversion $\Psi_g$
can be expressed as a graph near $(0,0,0)$.
Since $0<a<1$, the function $g$ satisfies (a) and (b) of
Proposition \ref{prop:a}.
Then $Z=Z_f(X,Y)$ as in \eqref{eq:zz} with $f:=g$
is a $C^1$-function at $(0,0)$.
The graph of $Z_g$ for $g=f_5$
near $(0,0,0)$
is indicated in the right-hand side of
Figure \ref{fig:01}.
To prove Theorem \ref{thm:C1},
 it is sufficient to
show that 
$(0,0,0)$ is a $C^1$-umbilic of the
graph of $Z_g(X,Y)$ with index $1+(m/2)$.
In the following discussions, 
we would like to show that
there exists a positive number $R$ such that
the graph of $g$ has no umbilics
if $r>R$. We then
compute the index  $I_g(\Gamma)$
with respect to the 
circle
\begin{equation}\label{eq:Gamma}
\Gamma(\theta):=(r \cos \theta,r \sin \theta)
\qquad (0\le \theta\le 2\pi,\,\, r>R),
\end{equation}
using \eqref{eq:general} and \eqref{eq:count},
which does not depend on the choice of $r(>R)$,
as follows. We set
\begin{equation}\label{eq:delta12}
\check \delta_j(\theta):=\delta_j(\Gamma(\theta))
\qquad
(j=1,2).
\end{equation}
The first polar identifier is given by
\begin{equation}\label{eq:delta1g}
\delta_1=
-m r^a s_m \biggl(a r^a c_m F''\left(r^a c_m\right)
+(a-1) F'\left(r^a c_m\right)\biggr).
\end{equation}
Since $0<a<1$, the condition (ii)
yields that
\begin{equation}\label{eq:F1}
(a-1) F'\left(r^a c_m\right)<0.
\end{equation}
On the other hand, by (i) and (iii),
it holds that
\begin{equation}\label{eq:F2}
xF''(x) \le 0 \qquad (x:=r^ac_m).
\end{equation}
By
\eqref{eq:F1} and \eqref{eq:F2},
we can conclude that
$\check \delta_1(\theta)$ changes sign
only at the zeros of the
function $\sin m\theta$.
Since the function $g$ is symmetric with
respect to rotation by angle $2\pi/m$,
to compute the rotation index of $\Delta_g$
along $\Gamma$, it is sufficient to 
check the sign changes of $\check\delta_i(\theta)$
($i=1,2$) for $\theta=0$ and $\theta=\pi/m$.
By
\eqref{eq:delta1g},
\eqref{eq:F1} and \eqref{eq:F2},
we get the following:
\begin{equation}\label{eq:d1-f}
\left.\frac{d\check \delta_1}
{d\theta}\right|_{\theta=0}>0,\qquad
\left.\frac{d\check \delta_1}
{d\theta}\right|_{\theta=\pi/m}<0.
\end{equation}
The second polar identifier
$\delta_2$ is given by
\begin{align*}
r^{2-3a}\delta_2
&=
-r^{2-a} \left(a^2 c_m^2-m^2 s_m^2\right) 
F''\left(c_m r^a\right) \\
&\phantom{aaaaaaa}+a c_m  \left(a^2 c_m^2-a m^2+m^2 s_m^2\right) F'\left(c_m r^a\right)^3\\
&\phantom{aaaaaaaaaaaa}
-c_m r^{2-2a} \left(a^2-2 a+m^2\right) F'\left(c_m r^a\right).
\end{align*}
We need the sign of $\check \delta_2(\theta)$
at $\theta\in (\pi/m)\Z$.
In this case, $s_m=0$ and $c_m=\pm 1$.
Substituting these relations
and using the fact that
$F'$ (resp. $F''$)
is an even 
function (resp. an odd function),
 we have
\begin{align*}
r^{2-3a}\delta_2
&=
\mp r^{2-a} a^2 
F''\left(r^a\right) \\
&\pm a^2 \left(a-m^2\right) F'\left(r^a\right)^3
\mp r^{2-2a} \left(a^2-2 a+m^2\right) 
F'\left(r^a\right).
\end{align*}
Since $F'$ is bounded, 
the middle term is bounded.
Hence, by (iv) and by 
the fact that $0<a<1$,
there exists a positive number $R$
such that
the sign of $\delta_2$ 
is determined by
the sign of the first term
$\mp r^{2-a} a^2 
F''\left(r^a\right)$ whenever $r>R$.
Then, we have
\begin{equation}\label{eq:pmd2}
-\check \delta_2(\pi/m)=\check \delta_2(0)>0.
\end{equation}
In particular, the image of the graph of $g$
has no umbilics when $r>R$.
By the $2\pi/m$-symmetry of $g$,
\eqref{eq:count},
\eqref{eq:d1-f}, and \eqref{eq:pmd2},
the index $\op{ind}_{\Gamma}(\Delta_g)$
is equal to $-m$.
Then the index of the curvature line flow
along $\Gamma$ is equal to $I_g(\Gamma)=1-m/2$
by \eqref{eq:general}.
Then after inversion,
the Poincar\'{e}-Hopf index formula
yields that
the index $I_0$ of the umbilic of $\Psi_g$ 
at the origin is
$$
I_0=2-I_g(\Gamma)=1+m/2.
$$
If we choose $F(x):=\tanh x$, 
then the function
$Z_g(X,Y)$ satisfies the properties of
Theorem \ref{thm:C1}.

\medskip
We next prove the corollary.
We set
\begin{equation}\label{eq:lambdaXY}
\lambda:=
\frac{Z\sqrt{1+Z_X^2+Z_Y^2}}
{1+\sqrt{1+Z_X^2+Z_Y^2}},
\end{equation}
where $Z:=Z_g$ is the function
given in \eqref{eq:zz}.
Suppose that 
$\lambda$ and $\lambda \nu$
are a $C^1$-function and a $C^1$-vector field
defined on a sufficiently small neighborhood
of $(X,Y)=(0,0)$, respectively, where
$\nu$ is a unit normal vector field of the graph of $Z_g$.
Then the map
$$
\Phi:(X,Y) \mapsto (\xi(X,Y),\eta(X,Y))
$$
given by \eqref{eq:Phi}
for $f=Z_{f_m}$ is a local
$C^1$-diffeomorphism,
and is real-analytic on $U\setminus\{(0,0)\}$.
Then the proof of Fact A.1 in the appendix
is valid in our 
situation, and we can conclude that
the eigen-flow of the Hessian matrix of
$\lambda(\xi,\eta)$ is equal to the
curvature line flow of the 
map $P(\xi,\eta)$ given by (A.8).
Since the image of $P(\xi,\eta)$ 
coincides with that of $\Psi_{f_m}(u,v)$,
we get the
proof of the corollary in the introduction.

Thus, it is sufficient to show that
$\lambda$ and $\lambda \nu$
are $C^1$ at $(X,Y)=(0,0)$.
By \eqref{eq:lambdaXY}, we have the following expression
\begin{equation}\label{eq:lnu}
\lambda\nu=\frac{(ZZ_X,ZZ_Y,-Z)}{1+\sqrt{1+Z^2_X+Z^2_Y}}.
\end{equation}
By \eqref{eq:lambdaXY} and \eqref{eq:lnu},
we can conclude that
$\lambda(X,Y)$ and $\lambda(X,Y)\nu(X,Y)$
are $C^1$ at $(0,0)$ if 
\begin{equation}\label{eq:limitF}
\lim_{(X,Y)\to (0,0)}Z Z_{XX}=
\lim_{(X,Y)\to (0,0)}Z Z_{XY}=\lim_{(X,Y)\to (0,0)}Z Z_{YY}=0
\end{equation}
hold.
So to prove the corollary, it is sufficient
to show \eqref{eq:limitF}.
It can be easily seen that all of
$r^{1-a}g_r,r^{-a}g_\theta,r^{2-2a}g_{rr}$,
$r^{1-2a}g_{r\theta}$ 
and $r^{-2a}g_{\theta\theta}$
are bounded functions 
on $\R^2\setminus \Omega_R$.
Since $0<a<1/4$,
Proposition 2.7 yields that
the map $(u,v)\mapsto (X,Y)=\Pi\circ \Psi_g(u,v)$ 
is a $C^2$-map.
Then \eqref{eq:limitF} is equivalent to
\begin{equation}\label{eq:limitF2}
\lim_{(u,v)\to (0,0)}Z Z_{uu}=
\lim_{(u,v)\to (0,0)}Z Z_{uv}=\lim_{(u,v)\to (0,0)}Z Z_{vv}=0.
\end{equation}
Since $Z=h/(k+1)$, \eqref{eq:limitF2} follows from \eqref{eq:00},
\eqref{eq:lim-k} and the fact that
$$
\lim_{\rho\to 0}\rho h_{uu}=
\lim_{\rho\to 0}\rho h_{uv}=
\lim_{\rho\to 0}\rho h_{vv}
=0.
$$

\section{An alternative proof of the main theorem}
\label{rmk:fin}

In the previous section, we have proved
Corollary \ref{cor:C1}.
However, it is natural to expect that
one can give an explicit description of the
function with the desired properties.
The function $\lambda$ given in \eqref{eq:lambdaXY} does not have a simple
expression. On the other hand, we will see that functions
\begin{equation}\label{eq:Lambda}
(\Lambda=)=\Lambda_m:=
r^{2}\tanh(r^{-a}\cos m\theta)
\qquad (m=1,2,3,\dots)
\end{equation}
satisfy (1) and (2) of
Corollary \ref{cor:C1} if $0<a<1$.
We set
$$
(\lambda:=)\lambda_m:=r^2F(r^{-a}\cos m \theta),
$$
where $\xi=r\cos \theta,\,\,\eta=r\sin \theta$,
and $F:\R\to \R$ is a function satisfying 
the properties (i--iv)
given in the beginning of Section 5.
Then $\Lambda_m$ is a special case of
$\lambda_m$ for $F(x):=\tanh x$.
It holds that
\begin{align*}
\lambda_r&=
r \biggl(2 F(r^{-a}c_m)-ac_m r^{-a} F'(r^{-a}c_m )\biggr),
\quad
\lambda_\theta:=
-m r^{2-a} s_m F'\left(r^{-a}c_m\right),\\
\lambda_{rr}&=2 F(r^{-a} c_m)+ar^{-2 a}c_m 
\biggl((a-3) r^a F'(r^{-a} c_m)+a c_m F''(r^{-a} c_m)\biggr), \\
\lambda_{r\theta}&=
ms_m r^{1-2 a}\biggl((a-2) r^a F'(r^{-a} c_m)+a c_m F''(r^{-a} c_m)\biggr), \\
\lambda_{\theta\theta}&=
-m^2 r^{2-2 a} \biggl(r^a c_m F'(r^{-a} c_m)
-s_m^2 F''(r^{-a} c_m)\biggr),
\end{align*}
where $c_m$ and $s_m$ are defined in \eqref{eq:cmsm}.
We set
$$
\zeta_1:=2(r \lambda_{r\theta}-\lambda_\theta),\qquad
\zeta_2:=-r^2 \lambda_{rr}+r \lambda_r+\lambda_{\theta\theta}.
$$
Then each component of
the vector field 
$\delta_\lambda:=\zeta_1\partial_x+\zeta_2\partial_y$
is an identifier for the eigen-flow of the Hessian matrix of $\lambda$
at the origin given in the introduction (cf. \eqref{eq:id1}).
By a direct calculation, we have
\begin{align*}
\zeta_1&=
2m r^{2-2 a} s_m \biggr(a c_m F''(r^{-a} c_m)+(a-1) r^a F'(r^{-a} c_m)\biggr),\\
\zeta_2&=
-r^{2-2 a}(a^2 c_m^2-m^2 s_m^2) F''(r^{-a}c_m)-(a^2-2 a+m^2)
r^{2-a} c_m F'(r^{-a} c_m).
\end{align*}
By the property (ii) of $F$, $(a-1)r^a F'(r^{-a}c_m)$ is negative, and
by (ii) and (iii), $c_m F''(r^{-a}c_m)$
is also negative.
So $\zeta_1$ is positively proportional to
$-s_m(=-\sin m\theta)$. In particular,
$\zeta_1$ vanishes only when $s_m=0$.
Moreover, for fixed $r$, it holds that 
$d\zeta_1/d\theta<0$ 
(resp. $d\zeta_1/d\theta>0$)
if $c_m=1$
(resp. $c_m=-1$).

On the other hand,
if $s_m=0$ and $r$ tends to zero, then
$c_m=\pm 1$ and
$F'(\pm r^{-a})$ and $F''(\pm r^{-a})$  
tend to zero with exponential order (cf. 
the condition (iv) for $F(x)$).
Therefore,  the leading term of $\zeta_2$ for small $r$ is 
$-r^{2-2 a}(a^2 c_m^2-m^2 s_m^2) F''(r^{-a}c_m)$.
Hence,
for a fixed sufficiently small $r$,
the function $\zeta_2$ is positive (resp. negative)
if $c_m=1$ (resp. $c_m=-1$).
Summarizing these facts, one can easily show that
the index of the vector field $\delta_\lambda$
at $o:=(0,0)$ is equal to $m$.
So the index of the eigen-flow of the
Hessian matrix of $\lambda$ at $o$
is equal to $1+m/2$ (cf. Appendix B). 
On can easily check that
$\lambda$ is a $C^1$-function at $o$ and
the function $\lambda$ satisfies (1) and (2)
of Corollary \ref{cor:C1}.
Since $\Lambda$ is a special case of $\lambda$,
we proved that $\Lambda$ satisfies the 
desired properties.

\begin{figure}[htb]
 \begin{center}
 \includegraphics[height=3.9cm]{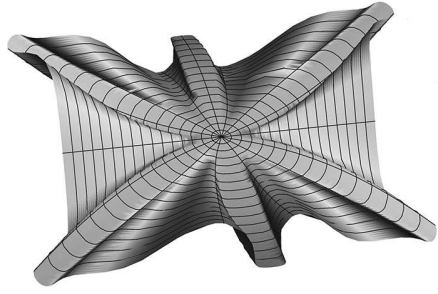}
\caption{The image of
$P$ ($r\le 1/2$)
for $m=2$ and $a=1/2$.
}
\label{fig:02}
 \end{center}
\end{figure}

To give an alternative proof of Theorem 1.1,
we consider the real analytic map $P:\R^2\setminus\{o\}
\to \R^3$ 
defined by (cf. \eqref{eq:P2})
$$
P(\xi,\eta):=(\xi,\eta,\Lambda(\xi,\eta))-\Lambda(\xi,\eta) 
\nu(\xi,\eta),
$$
where
\begin{equation}\label{eq:nu00}
\nu:=
\frac1{\Lambda_\xi^2+\Lambda_\eta^2+1}
(2\Lambda_\xi,2\Lambda_\eta,\Lambda_\xi^2+\Lambda_\eta^2-1).
\end{equation}
One can easily verify that
\begin{align*}
\Lambda_\xi&=
r^{1-a} \biggl((m s_1s_m-a c_1 c_m) \text{sech}^2\left(r^{-a} c_m\right)
+2 r^a c_1 \tanh \left(r^{-a} c_m\right)\biggr),\\
\Lambda_\eta&=
r^{1-a} \biggl(2 r^a s_1 \tanh \left(r^{-a} c_m\right)-(a s_1 c_m+m c_1 
s_m) \text{sech}^2\left(r^{-a} c_m\right)
\biggr),
\end{align*}
where $c_1=\cos \theta$ and $s_1=\sin \theta$.
Using them, one can get the following expressions
\begin{equation}\label{eq:hii}
\Lambda_{\xi\xi}=
\frac{1}{r^{2a}} h_1(r,\theta),\quad
\Lambda_{\xi\eta}=
\frac{1}{r^{2a}} h_2(r,\theta),\quad
\Lambda_{\eta\eta}=
\frac{1}{r^{2a}} h_3(r,\theta),
\end{equation}
where $h_i(r,\theta)$ ($i=1,2,3$)
are continuous functions
defined on $\R^2$.
Using \eqref{eq:nu00}, \eqref{eq:hii} and the fact $\lim_{r\to 0}\Lambda/r^{2a}=0$,
we have 
\begin{equation}\label{eq:L1}
\lim_{r\to 0}\Lambda \nu_{\xi}=\lim_{r\to 0}\frac{\Lambda}{r^{2a}}(r^{2a}\nu_{\xi})
=0,
\end{equation}
and also
\begin{equation}\label{eq:L2}
\lim_{r\to 0}\Lambda \nu_{\eta}=0.
\end{equation}
Using
\eqref{eq:L1}, \eqref{eq:L2}
and the fact 
$$
d(\Lambda \nu)=(d \Lambda) \nu+ \Lambda d\nu,
$$
we can conclude that
$\Lambda\nu$ can be extended as a $C^1$-function at $o$.
Thus  $P(\xi,\eta)$ can be also
extended as a $C^1$-differentiable map at $o$.
One can also easily check that
$$
P_\xi(0,0)=(1,0,0),\qquad P_\eta(0,0)=(0,1,0).
$$
Hence $P$ is an immersion at $o$, and 
$$
\Phi:(\xi,\eta)\mapsto (X(\xi,\eta),Y(\xi,\eta))
$$ 
is a local $C^1$-diffeomorphism,
where $P=(X,Y,Z)$.
In particular, 
$$
Z_\Lambda:=Z(\Phi^{-1}(X,Y))
$$ gives
a function  defined on a neighborhood of $(X,Y)=(0,0)$.
By Fact A.1 in the appendix,
the index of the curvature line flow at $(0,0)$
of the graph of $Z_\Lambda$ is equal to
the index of the eigen-flow of the Hessian matrix  of 
$\Lambda$, which implies Theorem \ref{thm:C1}.
The image of $P$ for $m=3$ and $a=1/2$ is 
given in Figure \ref{fig:02}.

\section{The duality of indices}
\label{rmk:fin2}
At the end of this paper, 
we consider the index at infinity
for eigen-flows of Hessian matrices.
Let
$$
f:\R^2\setminus \Omega_R\to \R,\qquad
 g:U_{1/R} \setminus \{ o\}\to \R
$$
be $C^2$-functions, 
where $\Omega_R$ and $U_{1/R}$
are disks defined in Section 2. 
Let $\mathcal H_f$ (resp. $\mathcal H_g$)
be the eigen-flow of the Hessian matrix of $f$
(resp. $g$). 
If the Hessian matrix of $f$ has no equi-diagonal points, then we can
consider the index ${\rm ind}\,(\mathcal{H}_f , \Gamma )$ with respect
to the circle $\Gamma$ given in (5.4) and it is independent of the
choice of $r>R$. So we denote it by $\op{ind}_\infty(\mathcal{H}_f)$.
Similarly, if the Hessian matrix of $g$ has no equi-diagonal points,
then we can consider the index ${\rm ind}\,(\mathcal{H}_g , \Gamma' )$
with respect to the circle $\Gamma'(\theta ):=(\rho \cos \theta , \rho
\sin \theta )$ ($0\leq \theta \leq 2\pi$, $\rho <1/R$). Since it is
independent of the choice of $\rho <1/R$,
we denote it by $\op{ind}_o(\mathcal{H}_g)$.
Consider the plane-inversion
$$
\iota:\R^2\in (u,v)\mapsto \frac1{u^2+v^2}(u,v)\in \R^2.
$$
Then the following assertion holds.

\begin{proposition}[The duality of indices]
\label{prop:d}
Let $f:\R^2\setminus \Omega_R\to \R$
be a $C^2$-function 
whose Hessian matrix
has no equi-diagonal points.
Then the function $g:\Omega_R\to \R$
defined by
$$
g(x,y):=(u^2+v^2)f\circ \iota(u,v)
$$
(called the {\it dual} of $f$)
 satisfies
$$
\op{ind}_o(\mathcal H_{g})+\op{ind}_\infty(\mathcal H_f)=2.
$$
\end{proposition}

\begin{proof}
Using the identification of $(u,v)$ and $z=u+iv$,
it holds that
$u=(z+\bar z)/2$ and $v=(z-\bar z)/(2i)$.
In particular, $f$ can be considered as 
a function of variables $z$ and $\bar z$,
and can be denoted by $f=f(z,\bar z)$.
Since $\iota(z)=1/\bar z$,
we can write
$
g(z,\bar z):=z \bar z f(1/\bar z,1/z).
$
Then 
$$
g_{zz}(z,\bar z)
=\frac{\bar z f_{\bar z \bar z}(1/\bar z,1/z)}{z^3}
$$
holds, where 
$$
\frac{\partial}{\partial z}:=
\frac12\left(\frac{\partial}{\partial u}
-i \frac{\partial}{\partial v} \right),\quad
\frac{\partial}{\partial \bar z}
:=\frac12 \left(\frac{\partial}{\partial u}
+i \frac{\partial}{\partial v}
 \right).
$$
Since $\Gamma(\theta)=re^{i\theta}$, we have that 
$$
g_{zz}(\Gamma(\theta))=\frac{f_{\bar z \bar z}
(\iota \circ \Gamma(\theta))}{r^{2}e^{4i\theta}}.
$$
Thus, it holds that
$$
\op{ind}_o(g_{zz},\Gamma)=-4+\op{ind}_o(f_{\bar z \bar z},\iota\circ \Gamma).
$$
By \eqref{eq:B1}, we have
\begin{align*}
\op{ind}_o(g_{zz},\Gamma)&=-2
\op{ind}_o(\mathcal H_{g}),\\
\op{ind}_o(f_{\bar z \bar z},\iota\circ \Gamma)
&=
-\op{ind}_o(f_{zz},\iota\circ \Gamma)=2
\op{ind}_\infty(\mathcal H_{f}).
\end{align*}
Thus we get the assertion.
\end{proof}

Applying Proposition \ref{prop:d}
for the function $g=\Lambda_m$
(cf.\eqref{eq:Lambda}), we get the 
following:

\begin{corollary}
For each $m(\ge 1)$,
there exists a $C^1$-function
$
f:\R^2\setminus \Omega_R\to \R
$
satisfying
\begin{enumerate}
\item $f$ is real-analytic on $\R^2\setminus \Omega_R$, 
\item the eigen-flow of the
Hessian matrix of $f$ has no
singular points,
and
\item the index at infinity of the eigen-flow of $H_f$
is equal to $1-m/2$.
\end{enumerate}
\end{corollary}

The function $\Lambda_m$ 
used in the second proof of Theorem 1.1 
coincides with the dual of the function $f_m-1$
given in \eqref{eq:tanh}.

\appendix
\section{The classical reduction}
In this appendix we show the existence of
a special coordinate system $(\xi,\eta)$ of
the graph of a function $f(x,y)$ which reduces
the curvature line flow to the Hessian of a
certain function, called Ribaucour's 
parametrization
(the third author learned this 
from Konrad Voss at the conference of
Thessaloniki 1997).
Although, the existence of such a 
coordinate system
was classically known, and a proof
is in the appendix of \cite{S}, the authors
will give the proof here for the
sake of convenience. 
We set
$
P=(x,y,f(x,y)),
$
and suppose that
$
f(0,0)=f_x(0,0)=f_y(0,0)=0
$.
Consider a sphere which is
tangent to 
the graph of $f$ at  $P$
and also tangent to the $xy$-plane at 
a point $Q$.
Then, it holds that
\begin{equation}\label{eq:fund}
Q+\lambda \mb e_3=P+\lambda \nu,
\end{equation}
where $\mb e_3=(0,0,1)$ and
$
\nu=(f_x,f_y,-1)/\sqrt{1+f_x^2+f_y^2}.
$
Taking the third component of
\eqref{eq:fund}, 
we get
\begin{equation}\label{eq:lambda0}
\lambda=\frac{f\sqrt{1+f_x^2+f_y^2}}{1+\sqrt{1+f_x^2+f_y^2}}.
\end{equation}
In particular,
$\lambda(0,0)=0$.
Since $f_x(0,0)=f_y(0,0)=0$, 
we have that
\begin{equation}\label{eq2}
d\lambda(0,0)=df(0,0)=0.
\end{equation}
Taking the exterior derivative of
\eqref{eq:fund}, and using \eqref{eq2}
and $\lambda(0,0)=0$,
we have $dP(0,0)=dQ(0,0)$.
So, 
if we set $Q=(\xi(x,y),\eta(x,y),0)$,
then it holds that
\begin{align*}
&(\xi_x(0,0)dx+\xi_y(0,0)dy,\eta_x(0,0)dx+
\eta_y(0,0)dy,0)
=dQ \\
&\qquad =dP=(dx,dy,f_x(0,0) dx+f_y(0,0)dy)=(dx,dy,0),
\end{align*}
which implies that
the Jacobi matrix of the map
\begin{equation}\label{eq:Phi}
\Phi:(x,y)\mapsto (\xi(x,y),\eta(x,y))
\end{equation}
is the identity matrix at $(0,0)$.
So we can take $(\xi,\eta)$ as 
a new local coordinate system.
Differentiating \eqref{eq:fund} by 
$\xi$ and $\eta$,
we get the following two identities:
$$
Q_{\xi}+\lambda_{\xi} \mb e_3
=P_{\xi}+\lambda_\xi \nu+\lambda \nu_{\xi}, 
\quad
Q_{\eta}+\lambda_{\eta} \mb e_3
=P_{\eta}+\lambda_{\eta} \nu+\lambda \nu_{\eta}.
$$
Taking the inner products of them and $\nu$,
these two equations yield
\begin{equation}\label{eq:Q}
Q_{\xi}\cdot \nu+\lambda_{\xi} \nu_3 
=\lambda_{\xi},  \qquad
Q_{\eta}\cdot \nu+\lambda_{\eta} 
\nu_3 =\lambda_{\eta}, 
\end{equation}
where we set
$
\nu=(\nu_1,\nu_2,\nu_3).
$
Since $Q=(\xi,\eta,0)$, we have that
$Q_{\xi}=(1,0,0)$ and $Q_{\eta}=(0,1,0)$.
So it holds that
$Q_{\xi}\cdot \nu=\nu_1$ and $Q_{\eta}\cdot \nu=\nu_2$.
Substituting this into 
\eqref{eq:Q}, we have
\begin{equation}\label{eq:lambda}
\lambda_{\xi}=\frac{\nu_1}{1-\nu_3},\qquad
\lambda_{\eta}=\frac{\nu_2}{1-\nu_3}.
\end{equation}
This implies that $(\lambda_{\xi},\lambda_{\eta})$ 
is the image of $\nu$ via the stereographic projection,
and we can write
\begin{equation}\label{eq:N}
\nu=\frac1{1+\lambda_{\xi}^2
+\lambda_{\eta}^2}(2 \lambda_{\xi},
2\lambda_{\eta},
\lambda_{\xi}^2+\lambda_{\eta}^2-1).
\end{equation}
By \eqref{eq:fund}, we have
\begin{equation}\label{eq:P2}
P=(\xi,\eta,0)-\lambda\nu+(0,0,\lambda).
\end{equation}
We prove the following

\begin{fact}
The curvature line flow of
the graph $z=f(x,y)$ coincides
with the eigen-flow of the 
Hessian of the function $\lambda(\xi,\eta)$
given by \eqref{eq:lambda0}.
\end{fact}

\begin{proof}
Noticing (A.8), we set
$$
\Delta_{(\xi,\eta)}:=\op{det}\pmt{\nu\\ dP\\ d\nu}
=\op{det}\pmt{\nu\\
d\xi,d\eta,d\lambda \\
d\nu}.
$$
Then this gives a map $\Delta_{(\xi,\eta)}:T_{(\xi,\eta)}\R^2\to \R$
such that
$$
\Delta_{(\xi,\eta)}\left(a\frac{\partial}{\partial \xi}
+b\frac{\partial}{\partial \eta}\right)
=\op{det}\biggl(\nu,
a P_{\xi}(\xi,\eta)+b 
P_{\eta}(\xi,\eta),
a \nu_{\xi}(\xi,\eta)+
b \nu_{\eta}(\xi,\eta)\biggr)\in \R.
$$
It is well-known that $\mb w\in T_{(\xi,\eta)}\R^2$
points in a principal direction of $P$
at $(\xi,\eta)$ if 
and only if $\Delta_{(\xi,\eta)}(\mb w)=0$.
Since
$(\nu_1)^2+(\nu_2)^2+(\nu_3)^2=1$,
\eqref{eq:lambda} yields that
$$
\lambda_{\xi} \nu_1+\lambda_{\eta} \nu_2
=\frac{(\nu_1)^2+(\nu_2)^2}{1-\nu_3}
=\frac{1-(\nu_3)^2}{1-\nu_3}=1+\nu_3,
$$
which implies
$
\nu_3=\lambda_{\xi} \nu_1+\lambda_{\eta} \nu_2-1.
$
We now set $\mu=2/(1+\lambda_{\xi}^2+\lambda_{\eta}^2)$.
Differentiating \eqref{eq:N}, we have
$$
d\nu=\frac{d\mu}{\mu}\nu+\mu (d\lambda_{\xi},
d\lambda_{\eta},
\lambda_{\xi}
d\lambda_{\xi}+\lambda_{\eta}d\lambda_{\eta}).
$$
The first term of the right hand-side of
the above equation is proportional to $\nu$
and does not affect
the computation of
$\Delta_{(\xi,\eta)}$.
So we have that
\begin{align*}
\Delta_{(\xi,\eta)}
&=\mu
\vmt{\nu_1 & \nu_2 & \lambda_\xi \nu_1+\lambda_\eta \nu_2-1 \\
  d\xi  &   d\eta &  \lambda_\xi d\xi+\lambda_\eta d\eta \\
  d\lambda_\xi & d\lambda_\eta &
 \lambda_\xi d\lambda_\xi+\lambda_\eta d\lambda_\eta} \\
&=
\mu
\vmt{\nu_1 & \nu_2 & -1 \\
         d\xi  &   d\eta &  0 \\
      d\lambda_\xi & d\lambda_\eta & 0}
=
-\mu\vmt{d\xi  &   d\eta  \\
      d\lambda_\xi & d\lambda_\eta} 
\\
&=\mu\biggl(
(\lambda_{\xi\xi}-\lambda_{\eta\eta})d\xi d\eta
-\lambda_{\xi \eta}(d\xi^2-d\eta^2)\biggr).
\end{align*}
Fact A.1 follows from this representation of $\Delta_{(\xi,\eta)}$.
\end{proof}

\section{Indices of eigen-flows of Hessian matrices}
Let
$
g:\Omega_R\setminus\{o\}\to \R
$
be a $C^2$-function, where
$\Omega_R$ is the closed disk 
of radius $R$ centered at the origin
 $o:=(0,0)$
(cf. \eqref{eq:Omega}).
The Hessian matrix of $g$ 
is given by 
$$
H_g:=\pmt{g_{xx} & g_{xy}\\ g_{yx} & g_{yy}}.
$$ 
We denote by $\mathcal H_g$ the
eigen-flow of $H_g$.
A point $p\in \Omega_R\setminus\{o\}$ 
is called an {\it equi-diagonal point}
of $\mathcal H_g$
if $H_g(p)$ is proportional to the identity matrix. 
Consider the circle
$$
\Gamma(\theta):=r (\cos \theta,\sin \theta)\qquad 
(0\le \theta<2\pi, r<R).
$$
If there are no equi-diagonal points on 
$\Omega_R\setminus\{o\}$, then
we can define the index 
$\op{ind}(\mathcal H_g,\Gamma)$
of the eigen-flow $\mathcal H_g$ 
with respect to $\Gamma$,
which does not depend on the choice of $r$.
We call it the index of $\mathcal H_g$ 
at the origin and denote it
by $\op{ind}_o(\mathcal H_g)$.
Consider the vector field
$$
d_g:=2g_{xy}\frac{\partial}{\partial x}+(g_{yy}-g_{xx})
\frac{\partial}{\partial y}.
$$
It is well-known that
the mapping degree of the Gauss map 
(cf. \eqref{eq:hat})
$$
\check d_g:
T^1:=\R/2\pi \Z\ni \theta 
\mapsto \frac{d_g(\Gamma(\theta))}{|d_g(\Gamma(\theta))|}\in
S^1:=\{(x,y)\in \R^2\,;\, x^2+y^2=1\}
$$
is equal to $2\op{ind}_o(\mathcal H_g)$.
Using the correspondence
$(x,y)\mapsto x+iy$,
we identify $\R^2$ with $\C$, 
where $i=\sqrt{-1}$.
Then
$$
g_{z}=\frac12(g_x-ig_y),\qquad
g_{zz}=\frac14((g_{xx}-g_{yy})-2 ig_{xy}),
$$
where $g_z:=\partial g/\partial z$,
$g_{zz}:=\partial^2 g/\partial z^2$
and 
$$
\frac{\partial}{\partial z}:=
\frac12\left(\frac{\partial}{\partial x}
-i
\frac{\partial}{\partial y}\right).
$$
Thus, $d_g$ can be identified with
the right-angle rotation of 
$\overline{g_{zz}}$.
In particular, we have
\begin{equation}\label{eq:B1}
\op{ind}_o(\mathcal H_g)
=-\frac12\op{ind}_o(g_{zz}).
\end{equation}
Here $g_{zz}$ is considered as a vector field
and $\op{ind}_o(g_{zz})$ is its index at the origin. 
Let $(r,\theta)$ be as in \eqref{eq:xy}.
Then $z=r e^{i\theta}$ and 
$$
g_{z}=\frac{e^{-i\theta}}{2r}(r g_r-ig_\theta),\quad
g_{zz}=
\frac{e^{-2i\theta}}{4r^2}
\biggl(
(r^2 g_{rr}-r g_r-g_{\theta\theta})+2i 
(g_\theta-r g_{r\theta})
\biggr).
$$ 
We consider the vector field defined by
\begin{equation}\label{eq:Delta2}
\delta_g:= 
2(r g_{r\theta}-g_\theta)\frac{\partial}{\partial x}
+(-r^2g_{rr}+r g_r-g_{\theta\theta})
\frac{\partial}{\partial y}.
\end{equation}
Since
$$
\op{ind}_o(\overline{g_{zz}})=2 +\op{ind}_o(\delta_g),
$$
we obtain the following:

\begin{lemma}\label{fact:k2}
The identity
$\op{ind}_o(\mathcal H_g)=1+\op{ind}_o(\delta_g)/2$
holds.
\end{lemma}

\begin{ack}
The third author thanks Udo Hertrich-Jeromin
for fruitful conversations.
The authors thank Wayne Rossman for
valuable comments.
\end{ack}

\end{document}